\documentclass[9pt]{amsart}

\usepackage{amsmath,amsthm,mathrsfs}
\usepackage{amssymb,amsfonts}
\usepackage{url,relsize,xcolor}
\usepackage[hidelinks]{hyperref}

\theoremstyle{plain}
\newtheorem{theorem*}{Theorem}
\newtheorem{theorem}{Theorem}[section]
\newtheorem{proposition}[theorem]{Proposition}
\newtheorem{lemma}[theorem]{Lemma}

\newtheorem{claim}[theorem]{Claim}

\numberwithin{equation}{section}

\theoremstyle{definition}
\newtheorem{definition}[theorem]{Definition}

\theoremstyle{remark}

\theoremstyle{theorem}

\newtheorem*{mainthm}{Theorem}

\newcommand{\phii}{\varphi}
\newcommand{\liftphi}{\Phi}
\newcommand{\bbN}{{\mathbb N}}

\newcommand{\cW}{{\mathcal W}}
\newcommand{\cF}{{\mathcal F}}
\newcommand{\cG}{{\mathcal G}}

\newcommand{\cP}{{\mathcal P}}

\newcommand{\cX}{{\mathcal X}}

\newcommand{\cB}{{\mathcal B}}
\newcommand{\cA}{{\mathcal A}}

\newcommand{\cJ}{{\mathcal J}}

\newcommand{\twoN}{\{0,1\}^\bbN}

\newcommand{\cstar}{$\mathrm{C^*}$}
\newcommand{\cPN}{\cP(\bbN)}

\newcommand{\cJprod}{\cJ_{\textrm{prod}}}

\DeclareMathOperator{\OCA}{{\mathsf {OCA}}}

\DeclareMathOperator{\OCAsharp}{{\mathsf {OCA}^\#}}

\DeclareMathOperator{\CH}{{\mathsf{CH}}}

\DeclareMathOperator{\MA}{{\mathsf {MA}_{\aleph_1}}}
\DeclareMathOperator{\MAsl}{{\mathsf {MA}_{\aleph_1}(\sigma\textrm{-linked})}}

\DeclareMathOperator{\Diff}{Diff}
\DeclareMathOperator{\Aut}{Aut}
\DeclareMathOperator{\dom}{dom}

\DeclareMathOperator{\Fin}{Fin}

\newcounter{my_enumerate_counter}
\newcommand{\pushcounter}{\setcounter{my_enumerate_counter}{\value{enumi}}}
\newcommand{\popcounter}{\setcounter{enumi}{\value{my_enumerate_counter}}}

\usepackage{enumitem}

\title{A metric lifting theorem}

\author{Ben De Bondt}
\address[BDB]{Universität Münster\\
	Institut für Mathematische Logik und Grundlagenforschung\\ 
	Einsteinstr.\ 62\\
	48149 Münster,
	Germany}
\email{bdebondt@uni-muenster.de}
\urladdr{}
\thanks{BDB is funded by the Deutsche Forschungsgemeinschaft (DFG, German Research Foundation) under Germany's Excellence Strategy EXC 2044\,-390685587, Mathematics Münster: Dynamics--Geometry--Structure}

\author{Alessandro Vignati}
\address[AV]{
 Institut de Math\'ematiques de Jussieu - Paris Rive Gauche (IMJ-PRG)\\
 Universit\'e Paris Cit\'e and Institut Universitaire de France,
 B\^atiment Sophie Germain\\
 8 Place Aur\'elie Nemours \\ 75013 Paris, France}
\email{ale.vignati@gmail.com}
\urladdr{http://www.automorph.net/avignati}
\thanks{AV is supported by the Institut Universitaire de France}

\date{\today}

\begin{document}	

\begin{abstract}
In the recent article \cite{TrivIso}, Farah and the authors proved a strong lifting theorem for well-behaved maps between reduced products of discrete structures, under the assumption of fairly mild Forcing Axioms. In this note, we prove the analogue of this result in the metric setting. 
\end{abstract}
	
\maketitle
\section{Introduction}
If $(M_n,d_n)$ is a sequence of metric spaces, we consider the equivalence relation $\sim_{\Fin}$ on $\prod_n M_n$, given by $a\sim_{\Fin} b$ if and only if $\lim_n d_n(a_n,b_n)=0$. This equivalence relation corresponds in the metric setting to `eventual equality' in the discrete setting. Moreover, in case the spaces $M_n$ have uniformly bounded diameter, quotienting by this equivalence relation gives the reduced product metric space $\prod_n M_n/\Fin$, measuring asymptotic distances between sequences. 
Just as metric ultraproducts, such metric reduced products are useful in several applications, for example in \cstar-algebra theory (see \cite{EffrosRosenberg} or \cite{Farah.Between}). 
Reduced products can more generally serve as a tool to study approximability of certain larger objects by means of simpler finite or countable objects.

Our main object of study in the present note is maps between two such reduced products $\prod_nM_n/\Fin$ and $\prod_n N_n/\Fin$, where $(N_n,\partial_n)$ is a second sequence of metric spaces. 
A natural and straightforward way to construct such maps is as follows. Consider a bijection $f$ between two cofinite subsets of $\bbN$ (i.e., an \emph{almost permutation} of $\bbN$, which plays the role of shuffling the indices) and a sequence of maps $h_n\colon M_{f(n)}\to N_n$ for $n \in \bbN$. If the pair $(f,(h_n))$ respects asymptotic equality, meaning that whenever $a,b\in\prod_nM_n$ are such that $\lim_n d_n(a_n,b_n)=0$, we have that \[ \lim_n\partial_n(h_n(a_{f(n)}),h_n(b_{f(n)}))= 0,\]
the function $\Phi[f,(h_n)]\colon\prod_nM_n\to\prod_nN_n$ defined by sending $a = (a_n)\in \prod_nM_n$ to $b= (b_n) \in \prod_n N_n$ where $b_n = h_n(a_{f(n)})$ induces a map between the corresponding reduced products over $\Fin$. Maps $\Phi\colon\prod_nM_n\to\prod_nN_n$ of this form $\Phi = \Phi[f,(h_n)]$ are called \emph{of product form}, and if $\varphi\colon \prod_n M_n/\Fin\to\prod_nN_n/\Fin$ admits a lifting of product form, we say that $\varphi$ is \emph{trivial}. Here $\Phi$ is called a \emph{lifting} of $\varphi$ if for every $a\in \prod_nM$, the $\sim_{\Fin}$-equivalence class containing $a$ is mapped by $\phii$ to the class containing $\Phi(a)$.

A key feature of trivial maps from $\prod_nM_n / \Fin$ to $\prod_nN_n / \Fin$ is that they respect a family of equivalence relations induced by pseudometrics, denoted $d^S(.\,,.\,)$ indexed by $S \in \mathcal P(\mathbb N)/\Fin$ which generalise to the metric setting the equivalence relation given by `being eventually equal on indices in $S$'. Maps preserving these equivalence relations are called \emph{coordinate respecting}. In case all the structures under consideration are countable sets with the discrete metric, these notions agree with their discrete counterparts defined in \cite[\S2]{TrivIso}.

The fact that reduced products of metric structures are countably saturated, together with the fact that the space of theories in a given separable language is compact, gives that, under the Continuum Hypothesis $\CH$, there exist metric spaces $M_n$ and $N_n$ and coordinate respecting functions $\prod_nM_n/\Fin\to\prod_nN_n/\Fin$ which are not trivial (see, \cite[Proposition 6]{TrivIso} in case all structures are discrete, or \cite{ghasemi2014reduced} for a truly metric version of this fact). The following, our main result, shows this is not the case under suitable Forcing Axioms.

\begin{mainthm} \label{th.mainIntro}
Assume $\OCA$ and $\MAsl$. Then all coordinate respecting functions between reduced products of sequences of separable metric spaces with uniformly bounded diameter are trivial.
\end{mainthm}

Here $\OCA$ is the Open Colouring Axiom, and $\MAsl$ is the restriction of $\MA$ to $\sigma$-linked partial orders. We refer the reader to \cite{farah2022corona} and \cite{TrivIso} for a thorough discussion on the influence of these axioms on rigidity of massive quotient structures.

The proof of the main result stays close to that of \cite[Theorem 7]{TrivIso}, with a few small technical detours needed to handle the presence of distances. 

First, we prove this result for sequences of finite metric spaces: there, having fixed a coordinate respecting function $\phii$, we show that (under Forcing Axioms) there is a nonmeager ideal $\cJprod\subseteq\mathcal P(\mathbb N)$ on whose sets $A \in \cJprod$  we can construct liftings of product form (i.e., induced by a sequence of maps $h_n$ for $n\in A$). We then uniformise all such maps. Secondly, we once again uniformise (once more using crucially $\OCA$) to show our rigidity results for all reduced products of uniformly bounded sequences of separable metric spaces.

\section{The main result}
Let $(M_n,d_n)$ be a sequence of separable metric spaces of uniformly bounded diameter.
The metric reduced product $M:=\prod_n M_n/\Fin$ is the metric space of $\sim_{\Fin}$-equivalence classes, where, for $a=(a_n)$ and $b=(b_n)$ in $\prod_nM_n$ we write 
\[
a \sim_{\Fin} b \text{ if and only if } \lim_{n} d_n(a_n,b_n)=0.
\] 
If $a=(a_n)\in\prod_nM_n$ we write $[a]$ for its class in $M$. $M$ is a metric space, where the distance is given by $d_M([a],[b])=\limsup_n d_n(a_n,b_n)$. Every infinite $S\subseteq\mathbb N$ induces a pseudometric $d^S$ on $M$ by setting, for $a=(a_n)$ and $b=(b_n)$ in $\prod_nM_n$,
\[
d^S([a],[b])=\limsup_{n\in S} d_n(a_n,b_n).
\]
If $S$ is finite, we simply set $d^S$ to be $0$. 
 Note that $d^\mathbb N$ is precisely the metric $d_M$. Further, notice that if $S$ and $S'$ have the same image in $\mathcal P(\bbN)/\Fin$, then $d^{S}$ and $d^{S'}$ coincide (if the diameters of the factor spaces of our reduced products are uniformly bounded away from zero, this implication is an equivalence). We can thus refer to, when convenient, the pseudometric $d^S$ for $S\in\mathcal P(\bbN)/\Fin$. If $a$ and $b$ are in $M$, we write $a=_Sb$ for $d^S(a,b)=0$. 

\begin{definition}
Let $(M_n,d_n)$ and $(N_n,\partial_n)$ be two sequences of separable metric spaces of uniformly bounded diameter. Let $M=\prod_nM_n/\Fin$ and $N=\prod_nN_n/\Fin$. A function
$\phii\colon M\to N$ is \emph{coordinate respecting} if there is $\alpha\in \Aut(\mathcal P(\bbN)/\Fin)$ such that for all $S\in\mathcal P(\bbN)/\Fin$ and all $a,b\in M$ we have that
\[
a=_Sb\Rightarrow \phii(a)=_{\alpha(S)}\phii(b).
\]
\end{definition}

As mentioned in the introduction, examples of these maps can be obtained starting from an almost permutation of $\mathbb N$ and applying coordinatewise maps~$h_n$. 
\begin{definition}
Let $(M_n,d_n)$ and $(N_n,\partial_n)$ be two sequences of separable metric spaces of uniformly bounded diameter. Let $M=\prod_nM_n/\Fin$ and $N=\prod_nN_n/\Fin$. A function
$\phii\colon M\to N$ is \emph{trivial} if there is $f$, an almost permutation of $\bbN$, and maps $h_n\colon M_{f(n)}\to N_n$ such that for all $a=(a_n)\in\prod_nM_n$ we have that
\[
\phii([a])=[b] \text{ where } b_n=h_n(a_{f(n)}).
\]
\end{definition}
To have that the function $a=(a_n)\mapsto (h_n(a_{f(n)}))$ is a well-defined map between $M$ and $N$ we must have that whenever $a,b\in\prod_nM_n$ are such that $[a]=[b]$ we have that $\lim_n\partial_n(h_n(a_{f(n)}),h_n(b_{f(n)}))= 0$, i.e., $[h_n(a_{f(n)})]=[h_n(b_{f(n)})]$. Clearly, every trivial map $\phii\colon \prod_nM_n/\Fin \to \prod_nN_n/\Fin$ is coordinate respecting. Also, as mentioned in the introduction, $\CH$ provides $2^{\mathfrak c}$ coordinate respecting nontrivial functions.
The following is our main result.
\begin{theorem}\label{th.main}
Assume $\OCA$ and $\MAsl$. Then all coordinate respecting functions between reduced products of sequences of separable metric spaces with uniformly bounded diameter are trivial.
\end{theorem}

The rest of this article is dedicated to the proof of Theorem~\ref{th.main}. 
We fix:
\begin{itemize}
\item $(M_n,d_n)$ and $(N_n,\partial_n)$ separable metric spaces of uniformly bounded diameter. We let $M=\prod_nM_n/\Fin$ and $N=\prod_nN_n/\Fin$. We further let $d_M$ and $\partial_N$ be the metrics on $M$ and $N$. Similarly, we let $d^S$ and $\partial^S$, for $S\subseteq\mathbb N$, be the pseudometrics defined above on $M$ and $N$ respectively. 
\item $\varphi\colon M\to N$, a coordinate respecting function as witnessed by $\alpha\in\Aut(\cP(\bbN)/\Fin)$, and we let $\Phi\colon\prod_nM_n\to\prod_nN_n$ be a lifting for $\varphi$.
\end{itemize}

In the proof of Theorem~\ref{th.main}, it will suffice to concentrate on the case in which the automorphism $\alpha$ is the identity and prove that $\varphi$ has a lifting of product form. To this purpose, it is useful to consider the following sets $\cJprod^\varepsilon$. 
For $\varepsilon\geq 0$, define $\cJprod^\varepsilon$ as the set of all $A\subseteq\mathbb N$ for which there exists a sequence $(h^A_n)$ of maps $h^A_n\colon M_n\to N_n$ such that for all $a\in \prod_nM_n$ we have that
\[
\limsup_{n\in A}\partial_n(\Phi(a)_n,h_n(a_n))\leq\varepsilon.
\]
We call such a sequence $(h_n^A)_n$ an $\varepsilon$-lifting of product form on $A$. 
Writing $\cJprod$ for $\cJprod^0$, we notice that, when $\alpha$ is the identity, $\phii$ is trivial precisely when there is a lifting of product form on $\bbN$, i.e., when $\bbN\in\cJprod$.

\begin{lemma}\label{L.intersection}
If $\alpha$ is the identity, each $\cJprod^\varepsilon$ is an ideal containing all finite sets, and 
\[
\cJprod=\bigcap_{\varepsilon>0}\cJprod^\varepsilon.
\]
\end{lemma}
\begin{proof}
It is clear that each finite set is in $\cJprod^\varepsilon$, as $\partial^A$ is $0$ whenever $A$ is finite. Fix $\varepsilon\geq 0$, and assume that $A,B\in\cJprod^\varepsilon$, as witnessed by $(h_n^A)$ and $(h_n^B)$. Setting $h_n^{A\cup B}$ to be equal to $h_n^A$ if $n\in A$ and $h^B_n$ otherwise witnesses that $A\cup B\in\cJprod^\varepsilon$.

As the inclusion $\cJprod\subseteq \bigcap_{\varepsilon>0}\cJprod^\varepsilon$ is obvious (if $\delta<\varepsilon$ then $\cJprod^\delta\subseteq\cJprod^\varepsilon$), we are left to prove that all $A\in \bigcap_{\varepsilon>0}\cJprod^\varepsilon$ belong to $\cJprod$. Fix such an $A$, and let $(h_n^m)_{n,m\in\bbN}$ be maps where $h_n^m\colon M_n\to N_n$ and such that for every $m\geq 1$ the sequence $(h_{n}^m)_n$ is a $1/m$-lifting of product form on $A$.

Since pseudometrics respect the triangle inequality, for every $a=(a_n)\in \prod_nM_n$ and nonzero natural numbers $m$ and $m'$ we have that
\[
\partial^A([h_{n}^m(a_n)],[h_{n}^{m'}(a_n)])\leq \frac{1}{m}+\frac{1}{m'}.
\]
In particular, we can find an increasing sequence of natural numbers $k_m$, for $m\geq 1$ such that for all $1 \leq i,j\leq m$ and $n\geq k_m$ we have that
\[
\sup_{x\in M_n}\partial_n(h_n^i(x),h_n^j(x))\leq\frac{1}{i}+\frac{1}{j}.
\]
Let $k_{0}=0$ and $h_{n}^0\colon M_n\to N_n$ be any map. Setting $h_n=h_n^m$ whenever $n\in [k_m,k_{m+1})$ gives a sequence which witnesses that $A\in\cJprod$.
\end{proof}

\subsection{The finitary version}
We first prove the finitary version of our main result. In this section, we always work under the hypotheses of the following theorem.
\begin{theorem}\label{thm:mainfinite}
Assume $\OCA$ and $\MAsl$. Let $(M_n,d_n)$ be finite metric spaces of uniformly bounded diameter, and let $(N_n,\partial_n)$ be separable metric spaces. Let $M=\prod_nM_n/\Fin$ and $N=\prod_nN_n/\Fin$. Suppose that $\phii\colon M\to N$ is coordinate respecting. Then $\phii$ is trivial.
\end{theorem}
As $\OCA$ is assumed, by \cite[Theorem 1]{TrivIso}, all automorphisms of $\mathcal P(\bbN)/\Fin$ are induced by an almost permutation of $\mathbb N$, and so is the automorphism $\alpha=\alpha(\phii)$ given by the definition of coordinate respecting. Let $f$ be the almost permutation inducing~$\alpha$. Using $f^{-1}$ to re-index the domain does not change whether $\phii$ is trivial or not; we can thus suppose that $\alpha$ is the identity.
\begin{proposition}
Assume $\OCA$ and $\MAsl$. Then $\cJprod$ intersects all uncountable almost disjoint families, and so it is nonmeager.
\end{proposition}
\begin{proof}
 Let $\cA \subseteq \mathcal P(\bbN)$ be an uncountable almost disjoint family. We first prove that $\cJprod^{\varepsilon}$ intersects $\cA$ for any given $\varepsilon>0$. 

Fix $\varepsilon>0$, and recall that $\liftphi\colon\prod_nM_n\to\prod_nN_n$ is a fixed lifting of $\phii$. Define $\mathcal{X}=\prod_nM_n\times\cA$. 
Colour unordered pairs with elements in $\mathcal X$ by $\{(a,A),(b,B)\} \in K_0$ if and only if
\begin{enumerate}
	\item $A\neq B$ and 
	\item there is $n\in A\cap B$ such that $a_n=b_n$ yet $\partial_n(\Phi(a)_n,\Phi(b)_n) > \varepsilon$.
\end{enumerate}
Note that $K_0$ is an open subset of $[\cX]^2$ in the separable metrizable topology obtained when $\cX$ is viewed as a subset of $\prod_nM_n\times \cPN\times \prod_n N_n$, via the identification $(a,A) \mapsto (a,A,\Phi(a))$. (This is the main reason we assume each $M_n$ is finite, to avoid further complications in the definition of $K_0$, as this colouring is not open in case $M_n$ is not discrete.)

\begin{claim}
There exists no uncountable $K_0$-homogeneous subset of~$\cX$.
\end{claim}
\begin{proof} 
We argue by contradiction. Let $H\subseteq\cX$ be an uncountable $K_0$-homogeneous set. By a standard Martin's Axiom argument (see, for example, the beginning of the proof of \cite[Proposition 5.2]{TrivIso} which translates readily to our setting), we can shrink $H$ to an uncountable subset and assume that there exist $s^0,s^1\in\prod_n M_n$ such that 
\[ 
\forall (a,A)\in H\,\, \forall n\in A \,\,\,\,\,a_n\in\{s^0_n,s^1_n\}. 
\]
Fix $(a,A)\in H$. For $i=0,1$, let $I_a^i=\{n\in A\mid a_n=s_n^i\}$. Since $[a]=_{I_a^i}[s^i]$, then $[\liftphi(a)]=_{I_a^i}[\liftphi(s^i)]$, so we can find $n_a$ such that for each $i=0,1$, if $n\geq n_a$ and $n\in I_{a}^i$ then 
\[
\partial_n(\liftphi(a)_n,\liftphi(s^i)_n)\leq\varepsilon/2.
\]
As $H$ is uncountable, we can shrink it further and assume $n_a=n_b$ whenever $(a,A)$ and $(b,B)$ are in $H$. Further, by shrinking $H$ again, since each $N_n$ is separable, we can assume that whenever $(a,A)$ and $(b,B)$ belong to $H$ and $n\leq n_a$ then 
\[
\partial_n(\liftphi(a)_n,\liftphi(b)_n)\leq\varepsilon/2.
\]
This is the last refinement we need: pick distinct $(a,A)$ and $(b,B)$ in $H$, and $n$ which witnesses that $\{ (a,A),(b,B) \}$ satisfies the second requirement in the definition of $K_0$. Let $i$ be such that $n\in I_a^i$ (so that $n\in I_b^i$, as $a_n=b_n$). If $n\leq n_a$, then $\partial_n(\liftphi(a)_n,\liftphi(b)_n)\leq\varepsilon/2$, and if $n>n_a$ then
\[
 \partial_n(\liftphi(a)_n,\liftphi(b)_n)\leq \partial_n(\liftphi(a)_n,\liftphi(s^i)_n)+\partial_n(\liftphi(b)_n,\liftphi(s^i)_n)\leq\varepsilon.
 \]
 This is a contradiction.
\end{proof}
By applying $ \OCA$, let $\mathcal{X}=\bigcup_{n\in\bbN} \mathcal X_n$, where every $\mathcal X_n$ is $K_1$-homogeneous. For each~$n$ fix a dense countable subset $D_n$ of $\mathcal X_n$.
For every two natural numbers $k$ and $n$, pick a map $h_{n}^k\colon M_n\to N_n$ with the property that, for every $z\in M_n$,
\begin{center}\begin{minipage}{0.85\textwidth}
		if there exists $(b,B)\in X_k$ with $n\in B$ and $b_n=z$, then there exists $(b,B)\in D_k$ with $n\in B$, $b_n=z$ and $\partial_n(\liftphi(b)_n,h_{n}^k(z))<\varepsilon$.
\end{minipage} \end{center}
The construction of $h_{n}^k$ is possible by density of $D_k$ in $\mathcal X_k$ (again, since each $M_n$ is finite, we can require that $b_n$ actually equals $z$).
Define 
\[
\mathcal{A}_{\text{bad}}=\{ A\in\mathcal{A}: \exists a\ ( (a,A)\in\bigcup D_k )\}.
\]
\begin{claim}\label{cl:claim prod}
If $A\in \mathcal A\setminus \mathcal{A}_{\text{bad}}$ then there is $k$ such that for every $a=(a_n)\in \prod_{n} M_n$:
	\[
\partial^A (\phii([a]),[h_{n}^k(a_n)])\leq2\varepsilon.
	\]
	In particular, $A\in\cJprod^{2\varepsilon}$.
\end{claim}

\begin{proof}
	Suppose otherwise, and fix $A\in \mathcal{A}\setminus \mathcal{A}_{bad}$, $\delta > 0$ and elements $a(k)$ in $\prod_n M_n$ such that for every $k$
	\[
	Z_k=\{ n\in A: \partial_n(\liftphi(a(k))_n, h_n^k(a(k)_n)) > 2\varepsilon + \delta \}
	\]
	is infinite.
This is equivalent to the negation of the thesis since $\liftphi$ is a lift of $\phii$.

Let $Y_k\subseteq Z_k$ be infinite and pairwise disjoint, and fix $a\in\prod_n M_n$ such that $a$ equals $a(k)$ on $Y_k$, for every $k$. 
Let $k$ be such that $(a,A)\in \mathcal X_k$. Since $[a]=_{Y_k}[a(k)]$, then $[\liftphi(a)]=_{Y_k}[\liftphi(a(k))]$, hence there is $m$ such that if $n\in Y_k\setminus m$ then $\partial_n(\liftphi(a)_n,\liftphi(a(k))_n)<\delta$. 
Fix $n\in Y_k\setminus m$. Pick, by 
choice of the map $h^k_n$, an element $(b,B)\in D_k$ such that $b_n=a_n$, $n \in B$ and $\partial_n(\liftphi(b)_n,h_{n}^k(a_n))<\varepsilon$. Since $\mathcal X_k$ is $K_1$-homogeneous, we have that $\partial_n(\liftphi(b)_n,\liftphi(a)_n)<\varepsilon$. The triangular inequality gives that \[\partial_n(\liftphi(a(k))_n,h_{n}^k(a(k)_n))\leq 2\varepsilon+\delta,\] a contradiction to $n\in Y_k\subseteq Z_n$.
\end{proof} 
As $\varepsilon$ was arbitrary, each $\cJprod^\varepsilon$ intersects nontrivially all uncountable almost disjoint families, hence so does $\cJprod=\bigcap_{\varepsilon>0}\cJprod^{\varepsilon}$. Since every ideal containing all finite sets and intersecting all uncountable almost disjoint families is nonmeager (see \cite[Lemma 3.3.2]{Fa:AQ}), this concludes the proof.	
\end{proof}

Before moving on, we isolate a topological sufficient condition to obtain well-behaved approximate liftings.

If $\cA\subseteq\mathcal P(\bbN)$ and $n\in\bbN$, let $\cA^n=\{C\in\cA\mid n\in C\}$. We will use the following fact: if $\cA\subseteq\mathcal P(\bbN)$ is nonmeager then for all sufficiently large $n$, $\cA^n$ is nonmeager (see e.g. \cite[\S3.10]{Fa:AQ}).
\begin{lemma}
\label{L.Uniformization}
Suppose that $\cA,\cB\subseteq \cJprod$ are nonmeager in $\mathcal P(\bbN)$. Assume that for all sufficiently large $n$, all $x\in M_n$, and every two elements $y,z$ of $N_n$ with $\partial_n(y,z) > \varepsilon$, at least one of \[\{ C\in \cA^n\mid \partial_n(h_n^C(x),y) \leq \frac{\varepsilon}{3} \}\text{ and }\{ C\in \cB^n\mid \partial_n(h_n^C(x),z) \leq \frac{\varepsilon}{3} \}\]
is meager. Then $\bbN \in \cJprod^{3\varepsilon}$.
\end{lemma}

\begin{proof}
For every $C\in\cJprod$, fix a sequence $h^C=(h_n^C)$ where $h_n^C\colon M_n\to N_n$ and such that $\phii([a])=_C[h_n^C(a_n)]$ for all $a=(a_n)\in\prod_n M_n$.

By nonmeagerness of $\cA$ and $\cB$, both $\cA^n$ and $\cB^n$ are nonmeager for all sufficiently large~$n$. Thus, forgetting about finitely many indices, we can assume that this is true for all $n$. Let $E_n$ be a countable dense subset of $N_n$.

We will define two sequences $(\tilde h^\cA_n)$ and $(\tilde h^\cB_n)$ of maps from $M_n$ to $N_n$. Let $n\in\mathbb N$ and $x\in M_n$, we need to define $\tilde h^\cA_n(x)$ and $\tilde h^\cB_n(x)$. Since $E_n$ is dense and countable, we can find $y,z \in E_n$, such that $\{ C\in\cA^n\mid \partial_n(h_n^C(x), y) \leq \frac{\varepsilon}{3} \}$ and $\{ C\in\cB^n\mid \partial_n(h_n^C(x), z) \leq \frac{\varepsilon}{3} \}$ are nonmeager. By our hypothesis, we have that $\partial_n(y, z) \leq \varepsilon$. Define 
\[
\tilde h^\cA_n(x) = y\text{ and }\tilde h^\cB_n(x) =z.
\]
 Since both $M_n$ and $E_n$ are countable, there exists a meager $\mathcal M\subseteq \cA$ such that for every $n\in\mathbb{N}$, $x \in M_n$, and $w\in E_n$ with $\partial_n(w, \tilde h^\cB_n(x)) >\varepsilon$, we have 
	\[ \{ C\in\cA^n\mid \partial_n(h_n^C(x),w) \leq \frac{\varepsilon}{3} \} \subseteq \mathcal M. \]
 We have thus defined a sequence $(\tilde h^\cA_n)$ and a meager set $\mathcal M \subseteq \cA$ such that for every $C\in\cA \setminus \mathcal M$, $n\in C$ and $x\in M_n$ we have that
$
\partial_n(h_n^C(x),\tilde h_n^\cA(x)) \leq 3\varepsilon.
$

We claim that these functions witness that $\bbN\in\cJprod^{3\varepsilon}$. Indeed, if this is not the case, we can find $a\in\prod_nM_n$ and an infinite set $S\subseteq\bbN$ such that for all $n\in S$ one has $\partial_n(\liftphi(a)_n,\tilde h_n^\cA(a_n)) > 3\varepsilon+\delta$ for some $\delta>0$. Since $\cA\setminus \mathcal M$ is nonmeager, there would by \cite[Theorem~3.10.1(a)]{Fa:AQ} exist $C\in\cA\setminus \mathcal M$ such that $S\cap C$ is infinite. 

Since $\lim_{n\in S\cap C}\partial_n(\liftphi(a)_n, h_n^C(a_n))=0$ and $\limsup_{n\in S\cap C}\partial_n(h_n^C(x),\tilde h_n^\cA(x)) \leq 3\varepsilon$, the triangle inequality gives a contradiction.
\end{proof}
The last step of the proof of Theorem~\ref{thm:mainfinite} requires the notion of partial selectors for metric spaces, which is isolated and discussed in the appendix (see Appendix~\ref{Appendix}) to not interfere with the flow of the proof. 

\begin{proposition} \label{prop.BMliftcase}
Assume $\OCA$. If $\cJprod$ intersects every uncountable almost disjoint family, then $\mathbb N\in\cJprod$.
\end{proposition}
\begin{proof}
We prove that $\mathbb N\in \cJprod^\varepsilon$ for any $\varepsilon>0$. Again, for every $A\in \cJprod$ fix functions $h_n^A\colon M_n\to N_n$
such that $\phii([a])=_A[h_n^A(a_n)]$ for all $a=(a_n)\in\prod_n M_n$. These cohere:

\begin{claim}\label{C.Coherence} 
For all $A$ and $B$ in $\cJprod $ and $\varepsilon>0$ the set
\[ 
\Diff_\varepsilon(h^A,h^B)= \{n \in A \cap B : \exists a\in M_n\, (\partial_n(h^A_n(a),h_n^B(a))>\varepsilon)\}
\]
is finite. 
\end{claim}
\begin{proof} 
If not, we can find $A$ and $B$ in $\cJprod$, an infinite $C\subseteq A\cap B$, $\varepsilon>0$ and $a_n\in M_n$, for $n\in C$, such that $\partial^C([(h_n^A(a_n))],[(h_n^B(a_n))])\geq\varepsilon>0$. For $n\notin C$, let $a_n$ be any element of $M_n$ and set $a=(a_n)$. Since $h^A$ witnesses that $A\in\cJprod$ and $C\subseteq A$, then $\partial^C(\phii([a]),[h^A(a)])\leq\partial^A(\phii([a]),[h^A(a)])=0$. Similarly, $\partial^C(\phii([a]),[h^B(a)])=0$. This is a contradiction.
\end{proof}
Consider the separable metric space $(Y_n, d_n)$ with $Y_n= (N_n)^{M_n}$ and
$d_n(h_n,h'_n) = \max_{x \in M_n}\partial_n(h_n(x),h'_n(x))$
, so that $\{h^A\mid A\in \cJprod\}$ is a family of partial selectors (Definition~\ref{Def.PartialSelector}). Fix $\varepsilon>0$.

\begin{claim}\label{C.Uniformization.coherent}
There are sets $\cG_n$, for $n\in\bbN$, such that $\cJprod=\bigcup_n \cG_n$, and for all $n$ and all $A,B$ in $\cG_n$ we have $|\Diff_\varepsilon(h^A,h^B)|\leq n$.
\end{claim}
\begin{proof}
Assume otherwise. By $\OCA$ and Lemma~\ref{lem:partialselmetric} then there are a perfect tree-like almost disjoint family $\cA$, an uncountable $Z\subseteq \twoN$, and $f\colon Z\to \cJprod$ such that for every $A\in \cA$ and all distinct $x,y$ in $Z$ we have 
\[
(\Diff_\varepsilon(h^{f(x)},h^{f(y)})\cap A)\setminus \Delta(x,y)\neq \emptyset. 
\]
Since $\cJprod$ intersects every uncountable almost disjoint family, there is $B\in \cA\cap \cJprod$. For $x\in Z$ let $n(x)=\max (\Diff_{\varepsilon/2}(h^{f(x)},h^B))$ (with $\max\emptyset=0$). By refining $Z$ to an uncountable subset, we can assume $n(x)$ is constant for all $x\in Z$. Choose distinct $x,y$ in $Z$ satisfying $\Delta(x,y)\geq n$ and fix $j$ in $(\Diff_\varepsilon(h^{f(x)},h^{f(y)})\cap B)\setminus n$. Then for every $w\in M_j$ we have that 
\[
\partial_n(h^{f(x)}_j(w),h^{f(y)}_j(w))\leq \partial_n(h^{f(x)}_j(w),h^{B}_j(w))+\partial_n(h^{B}_j(w),h^{f(y)}_j(w))\leq\varepsilon.
\]
 This is a contradiction. 
\end{proof}

Since $\cJprod$ is nonmeager, there is $\bar n$ such that $\cG_{\bar n}$ as provided by Claim~\ref{C.Uniformization.coherent} is nonmeager. We want to show that $\bbN\in\cJprod^{9\varepsilon}$.

Let $\cA_0=\cB_0=\cG_{\bar{n}}$. By Lemma~\ref{L.Uniformization}, we have two options: either $\bbN\in\cJprod^{9\varepsilon}$, or there are arbitrarily large $k_1\in \bbN$ and $x_1\in M_{k_1}$, $y_1, z_1\in N_{k_1}$ with $\partial_{k_1}(y_1,z_1) > 3\varepsilon$ such that
\[
\cA_1:= \{ C\in \cA_0^{k_1}\mid \partial_{k_1}(h_{k_1}^C(x_1),y_1) \leq {\varepsilon} \} \text{ and }
\cB_1:= \{ C\in \cB_0^{k_1}\mid \partial_{k_1}( h_{k_1}^C(x_1) , z_1) \leq {\varepsilon} \}
\]
are nonmeager.

Suppose thus this second condition holds. By again applying Lemma~{\ref{L.Uniformization}}, but now to $\cA_1,\cB_1$, we can have that either $\bbN\in\cJprod^{9\varepsilon}$ or we can find $k_2>k_1$, $x_2\in M_{k_2}$ and $y_2,z_2\in N_{k_2}$ with $\partial_{k_2}(y_2,z_2) > 3\varepsilon$ such that
\[
\cA_2:= \{ C\in \cA_1^{k_2}\mid \partial_{k_2}(h_{k_2}^C(x_2),y_2) \leq {\varepsilon} \} \text{ and }
\cB_2:= \{ C\in \cB_1^{k_2}\mid \partial_{k_2}( h_{k_2}^C(x_2) , z_2) \leq {\varepsilon} \}
\]
are nonmeager.

We claim that this construction must stop. Suppose that $\mathcal A_{\bar n+1}$ and $\mathcal B_{\bar n+1}$ can be constructed. We thus have natural numbers $k_1<\ldots<k_{\bar{n}+1}$ and $\bar n +1$-tuples $(x_1,\ldots,x_{\bar{n}+1})$, $(y_1,\ldots,y_{\bar{n}+1})$ and $(z_1,\ldots,z_{\bar{n}+1})$, such that
\[x_i\in M_{k_i}, \text{ and } y_i, z_i\in N_{k_i} \text{ with } \partial_{k_i}(y_i,z_i) > 3\varepsilon \quad \forall i \in \{1 \ldots \bar n+1 \}\] and all the sets
\[ \cA_{k+1}=\{ C\in\cG_{\bar{n}}\mid \{k_1,\ldots,k_{\bar{n}+1} \}\subseteq C\text{ and }(\forall i\in \{1,\ldots,k+1 \}) \, \partial_{k_i}( h_{k_i}^C(x_i), y_i) \leq {\varepsilon} \}, \]
\[ \cB_{k+1}=\{ C\in\cG_{\bar{n}}\mid \{k_1,\ldots,k_{\bar{n}+1} \}\subseteq C\text{ and } (\forall i\in \{1,\ldots,k+1 \}) \,\partial_{k_i}(h_{k_i}^C(x_i),z_i) \leq {\varepsilon} \} \]
are nonmeager.

The latter is however impossible because any $A\in \cA_{\bar n+1}$, $B\in \cB_{\bar n+1}$ would satisfy $\{ k_1,\ldots,k_{\bar n+1}\}\subseteq \Diff_{\varepsilon}(A,B)$, in contradiction with $A, B \in \cG_{\bar n}$ and therefore $|\Diff_\varepsilon(A,B)|=\bar n$. This shows that $\bbN\in\cJprod^{9\varepsilon}$. As $\varepsilon$ is arbitrary and $\cJprod=\bigcap_{\varepsilon>0}\cJprod^{\varepsilon}$, we have the thesis.
\end{proof}

\subsection{The separable case}
We use the finitary version of our main result (Theorem~\ref{thm:mainfinite}) to prove Theorem~\ref{th.main}.

\begin{proof}[Proof of Theorem~\ref{th.main}]
We keep the notation as fixed right after the statement of Theorem~\ref{thm:mainfinite}. As before, we can assume that $\alpha(\varphi)$, the automorphism of $\mathcal P(\bbN)/\Fin$ given by the definition of coordinate respecting, is the identity. 

Fix a countable dense set $D_n\subseteq M_n$. Notice that every element of $\prod_nM_n/\Fin$ can be seen as an element of $\prod_n D_n/\Fin$: for $a=(a_n)\in\prod_nM_n$, just choose $b_n\in D_n$ with $d_n(a_n,b_n)<\frac{1}{n}$, so that $[a]=[b]$. Furthermore, if $\phii$ has a lifting of product form on $\bbN$ given by functions $h_n\colon D_n\to N_n$, one can extend these to functions from $M_n$ by setting, for $x\in M_n\setminus D_n$, $\tilde h_n(x)=h_n(y)$ where $y$ is any element of $M_n$ which has distance $\leq \frac{1}{n}$ from~$x$. These maps still provide a lifting for $\phii$, hence we can assume each $M_n$ is countable.

Enumerate each $M_n=\{x_{n,k}\}_{k\in \bbN}$ (allowing repetitions). 
For $n,m\in\bbN$, let $M_{n,m}=\{x_{n,i}\}_{i\leq m}$; each $M_{n,m}$ is a finite metric space. For $f\in\bbN^\bbN$, let 
\[
M_f=\prod_{n}M_{n,f(n)}/\Fin.
\]
For each $f\in \bbN^\bbN$, we apply Theorem~\ref{thm:mainfinite} to the functions $\phii_f=\phii\restriction M_f$, and get a sequence $(h_{n,f})_n$ such that $h_{n,f}\colon M_{n,f(n)}\to N_n$ and $h_f=(h_{n,f})_{n\in\bbN}$ witnesses that $\bbN\in\cJprod(\varphi_f)$, i.e., $h_f$ lifts $\phii$ on $M_f$. 
The following is the appropriate notion of coherence here.
\begin{claim}
Let $f,g\in\bbN^\bbN$ and $\varepsilon>0$. Then there is $n(f,g,\varepsilon)$ such that for all $n\geq n(f,g,\varepsilon)$ and $x\in M_{n,f(n)}\cap M_{n,g(n)}$ we have that
\[
\partial_n(h_{n,f}(x),h_{n,g}(x))<\varepsilon.
\]
\end{claim}
\begin{proof}
If $a=(a_n)\in M_f\cap M_g$ both $(h_f(a_n))$ and $(h_g(a_n))$ are liftings for $\phii([a])$. Hence for every such $a$ we have that $\limsup_n \partial_n(h_{n,f}(a_n),h_{n,g}(a_n))=0$.
\end{proof}
Fix $\varepsilon>0$ and colour unordered pairs in $[\bbN^\bbN]^2=K_0\sqcup K_1$ by $\{f,g\}\in K_0$ if and only if
\[
\exists n \exists x\in M_{n,f(n)}\cap M_{n,g(n)}\, (\partial_n(h_{n,f}(x),h_{n,g}(x))>\varepsilon).
\]
\begin{claim}
There is no uncountable $K_0$-homogeneous set.
\end{claim}
\begin{proof}
Suppose we have an uncountable $K_0$-homogenous $\mathcal H\subseteq\bbN^\bbN$. We may assume $\mathcal H$ is of size $\aleph_1$. Since we have $\OCA$, the bounding number\footnote{The bounding number $\mathfrak b$ is the size of the least $\leq_*$-unbounded set of functions in $\bbN^\bbN$.} $\mathfrak b$ is greater than $\aleph_1$, hence there is $g\in\bbN^\bbN$ such that $f\leq_* g$ for all $f\in \mathcal H$, where $\leq_*$ is the order of eventual domination on $\bbN^\bbN$. By shrinking $\mathcal H$ to an uncountable subset of it, we can assume that there is $\bar n$ such that for all $f\in \mathcal H$ we have $n(f,g,\varepsilon/2) \leq \bar n$ and, in addition, that for all $f,f'\in\mathcal H$ and $n\in \bbN$ we have that $f(n)= f'(n)$ when $n < \bar n$ and $f(n), f'(n) \leq g(n)$ when $n\geq\bar n$. Since all the spaces 
$N_n$ are separable we can also assume that for all $n\leq \bar n$, $x\in M_{n,f(n)}$ and $f,f'\in\mathcal H$ we have that $\partial_n(h_{n,f}(x),h_{n,f'}(x))<\varepsilon$. For any $f$ and $f'$ in $\mathcal H$, the triangle inequality now gives that $\{f,f'\}\notin K_0$. This is a contradiction.
\end{proof}
$\bbN^\bbN$ can be endowed with a separable metrizable topology that makes $K_0$ an open subset of the set $[\bbN^\bbN]^2$ of its unordered pairs. For this, declare a set $U \subseteq \bbN^\bbN$ to be a basic neighborhood of $f \in \bbN^\bbN$ if there exists $m \in \bbN$ and $\delta >0$ such that all $g \in \bbN^\bbN$ belong to $U$ which satisfy:
\[ \forall n \leq m \quad f(n) = g(n) \text{ and } \forall x \in M_{n,f(n)}\  \partial_n(h_{n,f}(x), h_{n,g}(x)) \leq \delta.\]
By $\OCA$, we can write $\bbN^\bbN=\bigcup\mathcal G_n$ where each $\mathcal G_{n}$ is $K_1$-homogeneous. As ${(\bbN^\bbN,\leq_*)}$ is $\sigma$-directed, there is $\bar n$ such that $\mathcal G_{\bar n}$ is $\leq_*$-cofinal. Since $\leq_*=\bigcup\leq_k$, where $\leq_k$ is the order given by $f\leq_k g$ if and only if for all $n\geq k$ one has that $f(n)\leq g(n)$, we can find $k$ such that $\mathcal G_{\bar n}$ is $\leq_k$-cofinal. 

We are ready to construct our functions: for $n\geq k$ and $x\in M_n$, we pick $g\in\mathcal G_{\bar n}$ such that $x\in M_{n,g(n)}$, and define $h_n(x)=h_{n,g}(x)$. For $n<k$ we let $h_n$ be any function. The functions $(h_n)_{n\in\mathbb N}$ witness that $\mathbb N\in\cJprod^{\varepsilon}$. This follows because for every $a \in \prod_n M_n,$ there exists $g \in \cG_{\bar n}$ such that $a_n \in M_{n,g(n)}$ for all $n\geq k$.  We then have that $[h_{n,g}(a_n)] = \varphi([a])$ and since, by $K_1$-homogeneity of $\cG_{\bar n}$,  $\partial_n(h_n(a_n), h_{n,g}(a_n)) \leq \varepsilon,$ for all $n\geq k,$ it follows that the functions $(h_n)_{n}$ provide an $\varepsilon$-lift of $\phii$. 

Since $\varepsilon$ is arbitrary and $\cJprod=\bigcap_{\varepsilon>0}\cJprod^{\varepsilon}$, we have the thesis.
\end{proof}

\appendix
\section{Metric partial selectors}\label{Appendix}
The following is generalising the content of \cite[\S3.1]{TrivIso}. 
\begin{definition}[{\cite[Definition 3.4]{TrivIso}}]\label{Def.PartialSelector}
A \emph{partial selector} for a sequence of sets $(Y_n)_{n\in \bbN}$, is simply a function $g$ with $\dom(g) \subseteq \bbN$ and $g\in \prod_{n\in \dom(g)} Y_n$. 
If $Y_n$ are topological spaces, and $\cF$ is a set of partial selectors for $(Y_n)_{n\in \bbN}$,  we topologize $\cF$ as follows. For every $n\in \bbN,$  fix a point $0_n\in Y_n$ and embed $\cF$ in $\prod_n Y_n \times \mathcal P(\bbN)$ via the map $g \mapsto (\hat g, \dom(g))$ where $\hat g\in \prod_n Y_n$ agrees with $g$ on $\dom(g)$ and satisfies $\hat g(n)=0_n$ for $n\in \bbN\setminus \dom(g)$. Now consider the subspace topology on $\cF$ induced by this embedding. 
\end{definition}

If the sets $Y_n$ are separable and metrizable spaces, then the induced topology on any family of partial selectors is again separable and metrizable.
If $g$ and $h$ are partial selectors for a sequence of metric spaces $(Y_n,d_n)$, and $\varepsilon>0$, we let 
\[
\Diff_\varepsilon(g,h)=\{n\in \dom(g) \cap \dom(h)\mid d_n(g(n),h(n))> \varepsilon\}.
\]
Below, we adapt the `partial selector dichotomy' (\cite[Proposition 3.7]{TrivIso}) to the metric setting. For distinct $x,y\in \{ 0,1\}^\mathbb{N}$, we denote by $\Delta(x,y)$ the minimal $n$ such that $x_n\neq y_n$.
\begin{lemma}\label{lem:partialselmetric}
Assume $\OCA$. Let $\cF$ be a family of partial selectors for a sequence of separable metric spaces $((Y_n,d_n))_{n\in\bbN}$, and let $\varepsilon > 0$. Then one of the following possibilities holds. 
\begin{enumerate}[label=$(\arabic*)$]
	\item \label{1.Unifm} There are $\cF_n$, for $n\in \bbN$, such that $\cF=\bigcup_n \cF_n$, and for all $n$ and all $g,h$ in $\cF_n$ we have $|\Diff_\varepsilon(g,h)|\leq n$. 
	\item \label{3.Unifm} There are a perfect tree-like almost disjoint family $\cA$, an uncountable $Z\subseteq \twoN$, and an injective function $f\colon Z\to \cF$ such that for every $A\in \cA$ and all distinct $x,y$ in $Z$ we have 
	\[
	(\Diff_\varepsilon(f(x),f(y))\cap A)\setminus \Delta(x,y)\neq \emptyset.
	\] 
\end{enumerate} 
\end{lemma}
In the following we use $\OCAsharp$, a formal strenghtening of $\OCA$ (\cite[Definition 3.2]{TrivIso}). 
\begin{proof}
	As the proof of \cite[Proposition 3.7]{TrivIso} goes through, mutatis mutandis, equally well for proving this metric generalization, we only sketch the required modifications.
One first defines colourings to which $\OCAsharp$ will be applied, and then shows that the first alternative of $\OCAsharp$ implies \ref{1.Unifm} while the second one implies \ref{3.Unifm}. Deriving \ref{1.Unifm} (resp.\ \ref{3.Unifm}) from the respective alternatives of $\OCAsharp$ proceeds exactly as in the discrete case (but considering the sets $\Diff_{\varepsilon}(g_0,g_1)$ instead of $\Diff(g_0,g_1)$). We therefore restrict to defining the sets of colourings $\cW^n=\{W^n_m\mid m\in \bbN\}$ to which to apply $\OCAsharp$.  

For any $l$, since $Y_l$ is a separable metric space, we can fix  open balls  $V(l,m,i)$, for $m\in \bbN$ and $i=0,1$, such that $d_l(y_0,y_1) > \varepsilon$ for all $y_0 \in V(l,m,0)$ and $y_1 \in V(l,m,1)$ and such that in addition for every pair $(y_0,y_1) \in Y_l^2$ with $d_l(y_0,y_1) > \varepsilon$ there is $m$ such that $y_i\in V(l,m,i)$ for $i=0,1$. 
	
	For every $n$, select a sequence $\langle S(n,m),u^n_{m}\rangle$, for $m\in \bbN$, that enumerates all pairs such that $S(n,m)\subseteq \bbN$, with $|S(n,m)|=n+(4^{n+1}-1)/3$ and $u^n_{m}\colon S(n,m)\to \bbN$. Let, for all $m$ and $i=0,1$, 
	\begin{multline*}
		U^n_{m,i}=\{g\in \cF\mid S(n,m)\subseteq \dom(g),\ 
		g(l)\in V(l,u^n_m(l),i) \text{ for all }l\in S(n,m)\}. 
	\end{multline*}
It is then easily checked that
	\[\label{eq.4n-1/3}
		\{(g_0,g_1)\in \cF^2\mid |\Diff_\varepsilon(g_0,g_1)|\geq n+(4^{n+1}-1)/3\} =\bigcup_m U^n_{m,0}\times U^n_{m,1}. 
	 \] 
	The relevant choices of $W^n_m$ and $\cW^n$ are now given by:
	\[
	W^n_m=U^n_{m,0}\times U^n_{m,1}\cup U^n_{m,1}\times U^n_{m,0} \text{ and }\cW^n=\{W^n_m\mid m\in \bbN\}.\qedhere
	\]
\end{proof}

\bibliographystyle{amsplain}
\bibliography{Bibliography_triv}
\end{document}